\documentclass{article}
\usepackage[utf8]{inputenc}

\usepackage{csquotes}
\usepackage[T1]{fontenc}
\usepackage{amsmath}
\usepackage{amssymb}
\usepackage{mathrsfs}
\usepackage{amsthm}
\usepackage{stmaryrd}
\usepackage{geometry}
\usepackage{setspace}
\usepackage{soul}
\usepackage{graphicx}
\usepackage{dsfont}
\usepackage{color}
\usepackage{listings}
\usepackage{xcolor}
\usepackage{float}
\usepackage{mathtools}
\usepackage{array}
\usepackage{biblatex}
\usepackage{enumitem}
\usepackage{pdfpages}

\usepackage{lipsum}

\usepackage{hyperref}
\usepackage[nameinlink,noabbrev]{cleveref}
\usepackage{makeidx}

\usepackage{xcolor}
\definecolor{XProcessBlue}{RGB}{0,33,102}

\hypersetup{
    colorlinks,
    citecolor=black,
    filecolor=black,
    linkcolor=black,
    urlcolor=black
}

\def\W{\mathcal W}

\def\R{\mathbb R}
\def\C{\mathbb C}
\def\N{\mathbb N}

\def\T{\mathbb T}

\def\L{\mathrm L}
\def\H{\mathrm H}

\def\eps{\varepsilon}

\def\Q{\mathrm Q}
\def \F{\mathcal{F}}

\def \p {\partial}
\DeclareMathOperator{\Tr}{Tr}

\DeclareMathOperator{\re}{Re}
\DeclareMathOperator{\im}{Im}

\DeclareMathOperator{\vspan}{span}

\def\d{\mathrm{d}}

\newcommand{\indic}[1]{\mathds{1}_{#1}}
\newcommand{\Atop}[2]{\genfrac{}{}{0pt}{1}{#1}{#2}} 

\newtheorem{theo}{Theorem}[section]
\crefname{theo}{Theorem}{Therorems}
\Crefname{theo}{Theorem}{Therorems}

\newtheorem{prop}[theo]{Proposition}
\crefname{prop}{Proposition}{Propositions}
\Crefname{prop}{Proposition}{Propositions}

\newtheorem{lemme}[theo]{Lemma}
\crefname{lemme}{Lemma}{Lemmas}
\Crefname{lemme}{Lemma}{Lemmas}

\crefname{defi}{Definition}{Definitions}
\Crefname{defi}{Definition}{Definitions}

\newtheorem{cor}[theo]{Corollary}
\crefname{cor}{Corollary}{Corollaries}
\Crefname{cor}{Corollary}{Corollaries}

\newtheorem{hyp}[theo]{Assumption}
\crefname{hyp}{Assumption}{Assumptions}
\Crefname{hyp}{Assumption}{Assumptions}

\crefname{section}{Section}{Sections}
\Crefname{section}{Section}{Sections}

\crefname{equation}{}{}
\Crefname{equation}{}{}

\crefname{remark}{Remark}{Remarks}
\Crefname{remark}{Remark}{Remarks}

\theoremstyle{definition}

\newtheorem{rem}[theo]{Remark}

\numberwithin{equation}{section}

\newcommand{\e}{\mathrm{e}} 
\newcommand{\complexI}{\mathrm{i}} 

\bibliography{biblio.bib}

\title{A simple derivation of Anderson-Hartree equations with harmonic potential}

\author{S. Mackowiak\thanks{Univ Brest,  CNRS UMR 6205,  Laboratoire de Mathématiques de Bretagne Atlantique,  F-29200,  Brest,  France} \thanks{Université de Lorraine,  CNRS,  IECL,  F-54000 Nancy,  France}\\
\href{mailto:samael.mackowiak@univ-lorraine.fr}{samael.mackowiak@univ-lorraine.fr}}

\date{}

\begin{document}

\maketitle

\begin{abstract}

    In this paper, we show how to adapt a derivation of Hartree equations due to Knowles and Pickl in the case where the one-particle hamiltonian is an Anderson-Hermite operator in dimension 1 and 2 or a formal Anderson-Hermite operator in dimension 3. By proving energy estimates for solution of the limit equation, uniform in the noise in bounded sets, we achieve a double limit. Namely, we show that we can commute the mean-field limit with a limit in noise.
    
\end{abstract}

\section{Introduction}

This paper aims to derive mathematically the Anderson-Hartree equation
\begin{equation}\label{Intro:AndersonHartreeEq}
    \complexI\p_t u +(H+\xi) u + w*|u|^2 u = 0
\end{equation}
from the many-body equation
\begin{equation}\label{Intro:ManyBodyEq}
    \complexI\p_t\Psi_N + \sum_{j=1}^N (H+\xi_\eps)_{x_j}\Psi_N + \frac{1}{N}\sum_{1\leqslant i<j\leqslant N}w(x_i-x_j)\Psi_N = 0,
\end{equation}
where $H$ is a one-particle hamiltonian on $\L^2(\R^d)$ ($d\in\{1,2\}$), $w$ is an interaction kernel, $\xi$ is a spatial white noise and $\xi_\eps$ is a smooth approximation of $\xi$. To do this, we have to deal with a double limit, the one in the number of particles $N$ and the one in the regularization parameter $\eps$. Our goal is to show that we can pass to the double limit in $N$ and $\eps$, thus allowing us to commute the limits. For the sake of simplicity, we consider $H=\Delta-|x|^2$ but we can adapt the result to more general one-particle hamiltonians of the form $H=\Delta-V$ with $V\approx |x|^\beta$. We extend our method to a class of formal Anderson-Hermite operators in dimension 3, in order to highlight the dimension dependent part of the approach. Let us point out we do not construct rigorously the 3d Anderson-Hermite operator in this article. We prove the following result. 

\begin{theo}\label{Th:MeanFieldLimitAndersonHartree}
    Let $d\in\{1,2,3\}$, $w\in\mathfrak{W}$ where $\mathfrak{W}$ is defined in \cref{Eq:DefW}, $N\geqslant 2$ and $v_0\in\W^{1,2}$.
    For $\theta=(\xi,Y,Z)\in\Theta$, $\Theta$ being defined by \cref{Def:ThetaAnderson}, let $-A_\theta$ the unbounded self-adjoint operator associated to the bilinear form \cref{Def:AndersonBilinearForm}. Let $u^{\theta,w}$ be the unique maximal solution of \begin{equation}\label{Eq:AndersonHartree}
        \begin{cases}
        \complexI\p_t u + A_\theta u + w*|u|^2 u = 0,\\
        u(0)=u_0^\theta = \e^{Y}v_0.
        \end{cases}
    \end{equation} 
     and $\Psi_{N,\theta,w}$ be the unique solution of 
    \begin{equation}
        \begin{cases}
        \displaystyle\complexI\p_t\Psi_N+\sum_{j=1}^N (A_\theta)_{x_j}\Psi_N + \frac{1}{N}\sum_{1\leqslant i<j\leqslant N}w(x_i-x_j)\Psi_N=0,\\
        \Psi_N(0)=\Psi_{N,\theta}^0=(u_0^\theta)^{\otimes N}.
        \end{cases}
    \end{equation}
    For $k\in\N^*$, let $\Gamma_{k,N,\theta,w}$ be the $k$-particles marginal associated to $\Psi_{N,\theta,w}$ defined by its kernel
    $$\forall x,y\in\R^k,\; \Gamma_{k,N,\theta,w}(x;y) = \int_{\R^{d(N-k)}}\Psi_{N,\theta,w}(x,z)\overline{\Psi_{N,\theta,w}(y,z)}\d z.$$
    Then, for any $k\in\N^*$, it holds
    $$\lim_{N\to+\infty}\Gamma_{k,N,\theta,w}(t) =\left(|u^{\theta,w}(t)\rangle\langle u^{\theta,w}(t)|\right)^{\otimes k}$$
    in trace-class uniformly in $(t,\theta,w)$ in bounded sets of $\R\times\Theta\times\mathfrak{W}$. In particular, one can commute all limits in $N$, $\theta$ and $w$ or pass to the limit in $N$ along sequences $(\theta_N,w_N)$ converging to $(\theta,w)$.
\end{theo}

This work is twofold. On the one hand, we want to derive Anderson-Hartree equations as effective nonlinear equations from the linear interacting many-body Schrödinger equation. If we put aside the Anderson part of the problem, we end up with the so-called quantum mean-field limit problem. For the past decades, there has been an extensive literature both in mathematics and physics to derive effective equations from quantum many-body problems. 

From a physical point of view, going from the many-body equation to a nonlinear Schrödinger type equation is quite natural in the context of Bose-Einstein condensation. In fact, when bosons condense, a macroscopic proportion of bosons occupy the same quantum state, deriving an effective equation then allows one to study the dynamics of this common quantum state. In the case of dilute gases, the limit equation is the so-called Gross-Pitaevskii equation, i.e. a cubic nonlinear equation. In this paper, we are interested in the easier problem to derive Hartree-type nonlinearities. This is sometimes called the "mean-field regime" or Hartree regime. See for example \cite{RougerieRewiewNBody} for a discussion of the different regimes.

From a mathematical point of view, the rigorous derivation of effective equations can be a tough problem. It leads over the last decades to the development of many different methods, each with their own strengths and drawbacks. Let us mention the BBGKY hierarchy method \cite{spohn1980kinetic,bardos2000weak,erdos2001derivation}, the coherent state method \cite{Hepp1974,RodnianskiQuantumFluctuations} and Pickl's method \cite{PicklKnowles,PicklSimple}. See also \cite{GolseReviex,BenedikterReviewMeanField,RougerieRewiewNBody} and references therein. 

In the stochastic setting, one can cite \cite{kolokoltsov2021law,kolokoltsov2022quantum,deBouardGuoHerouard2025} for quantum mean-field limits with additive noise and \cite{Zachhuber3dAndersonHartreeMeanField} for the mean-field derivation of Anderson-Hartree equations on $\T^3$. Up to the author's knowledge, \cite{Zachhuber3dAndersonHartreeMeanField} was the first paper to address the problem of quantum mean-field limits in the context of continuous Anderson operators. The authors use the BBGKY hierarchy and thus only prove mean-field limit to the Anderson-Hartree equation for bounded interactions as they cannot prove a useful uniqueness result for the solution of the limit hierarchy.\\

On the other hand, we aim to solve Anderson-Hartree equations of the form \cref{Intro:AndersonHartreeEq} and prove uniform bounds in some adapted functional spaces. In dimension 1, one can define directly the operator $H+\xi$ from its quadratic form. However, in dimension 2, a renormalization procedure is needed to construct the operator $H+\xi$, see \cite{Mackowiak_2025,MackowiakStrichartzConfAnderson} for the case $H=\Delta-|x|^2$ and see also \cite{allez2015continuous,Labb__2019,mouzard2021weyl,bailleul2022analysis,mouzard2023simple} for $H=\Delta$ on compact surfaces and \cite{UEKI2025104642,hsu2024construction} for $H=\Delta$ on $\R^2$. As already stressed out, the 3d part of this article relies on a formal construction of the Anderson-Hermite operator, i.e. we do not construct the stochastic objects involved in the renormalization process.

The case of the cubic Anderson-Gross-Pitaevskii equation, i.e. \cref{Intro:AndersonHartreeEq} with one-particle hamiltonian $H=\Delta-|x|^2$ and interaction potential $w=\lambda\delta_0$, has been treated in dimension 2 in \cite{Mackowiak_2025} by a compactness method for well-prepared regular initial data and in \cite{MackowiakStrichartzConfAnderson} for low regularity deterministic initial data, using Strichartz estimates. Let us also mention \cite{debussche_weber_T2,tzvetkov2020dimensional,Tzvetkov_2023} which deal with polynomial Anderson-NLS equation on $\T^2$ and \cite{debussche2017solution,debussche2023global} which deal with polynomial Anderson-NLS equation on $\R^2$.

In the case of Hartree non-linearities with interaction kernel $w\in\L^2+\L^\infty$, Strichartz estimates are not necessary in order to obtain finite energy solutions in dimension 1. In dimension 2 and 3, we introduce technical assumption on $w$, leading to the space $\mathfrak{W}$ of admissible interaction kernels given by \cref{Eq:DefW}. Let us point out that the technical assumption we made on $w$ is essentially the one needed to apply the result of \cite{PicklKnowles}. Global wellposedness and conservation laws for Hartree equation with quite general one particle hamiltonian and interaction kernel in $\mathfrak{W}$ are given in \cref{Cor:GWPQ(A)}. In particular, \cref{Cor:GWPQ(A)} provides an energy estimate sufficient to apply the result of \cite{PicklKnowles}.

Let us point out that even if we keep our method as simple as possible, i.e. without using any advanced dispersive tool, we are able to treat the physically relevant Coulomb potential $w(x)\propto |x|^{-1}$ in dimension 3.

\section{An abstract setting}

Let $A$ be a self-adjoint operator on $\L^2(\R^d)$ and $w\in\L^2+\L^\infty$. Denote by $\Q(A)$ the form domain of $A$ and by $\mathrm{D}(A)$ its domain. Define the many-body hamiltonian 
\begin{equation}\label{Eq:ManyBodyHamiltonian}
    H_{N} =\sum_{j=1}^N A_{x_j} + \frac{1}{N}\sum_{1\leqslant i<j\leqslant N} w(x_j-x_i).
\end{equation}
and the associated many-body Schrödinger equation 
\begin{equation}\label{Eq:ManyBodySchrodinger}
    \begin{cases}
    \complexI\p_t\Psi_N+H_{N}\Psi_N=0,\\
    \Psi_N(0)=\Psi_N^0.
    \end{cases}
\end{equation}
Formally, the mean-field limit is given by the Hartree equation
\begin{equation}\label{Eq:Hartree}
    \begin{cases}
    \complexI\p_t u + A u + w*|u|^2 u = 0,\\
    u(0)=u_0.
    \end{cases}
\end{equation}
This equation has formally a conserved energy given by
\begin{equation}\label{Eq:DefEnergyHartree}
    \mathcal{E}(u) = \frac{1}{2}\langle -A u,u\rangle - \frac{1}{4}\int|u|^2w*|u|^2\d x.
\end{equation}

For $w\in\mathcal{M}+\L^\infty$, let
\begin{equation}\label{Eq:Def[w]}
    [w] = \sup_{\Atop{u\in\Q(A)}{u\neq 0}}\frac{\left|w*|u|^2\right|_{\L^\infty}}{|u|_{\Q(A)}^2}\\
    = \sup_{\Atop{u_1,u_2\in\Q(A)}{u_1,u_2
    \neq 0}}\frac{\left|w*(u_1u_2)\right|_{\L^\infty}}{|u_1|_{\Q(A)}|u_2|_{\Q(A)}}.
\end{equation}
For $w\in\L^2+\L^\infty$, it is easy to see
\begin{equation}\label{Eq:Control[w]QL2}
    \forall u\in\Q(A),\; \left|w*|u|^2\right|_{\L^\infty} \leqslant \left[w^2\right]^{\frac{1}{2}}|u|_{\Q(A)}|u|_{\L^2},
\end{equation}
thus there exists $c=c(A)>0$ such that $[w]\leqslant c\left[w^2\right]^{\frac{1}{2}}$.\\

In what follows, we restrict ourselves to interaction kernels belonging to
\begin{equation}\label{Eq:DefW}
    \mathfrak{W}=\left\{w\in\L^2(\R^d,\R)+\L^\infty(\R^d,\R),\; w(x)=w(-x)\text{ a.e.},\; \left[w^2\right]<+\infty,\; \limsup_{\infty}|w|<+\infty\right\}.
\end{equation}
Note that $\mathfrak{W}$ is in fact a vector space as 
$$\forall \lambda\in\R,\; \left[(\lambda w)^2\right]=\lambda^2\left[w^2\right]\text{ and }\left[(w_1+w_2)^2\right]\leqslant 2\left(\left[w_1^2\right]+\left[w_2^2\right]\right).$$
Let us point out the symmetry assumption will not be used in the global wellposedness result but will be needed later on, to deal with mean-field limits. The condition $\left[w^2\right]<+\infty$ is exactly the one we need in order to apply the general mean-field convergence result of \cite{PicklKnowles}. Moreover, $\mathfrak{W}$ is a Hausdorff quasi-Banach space when endowed with the quasi-norm
$$|w|_{\mathfrak{W}} = |w|_{\L^2+\L^\infty}+\left[w^2\right]^{\frac{1}{2}}.$$

For $w\in\mathfrak{W}$, let $(w_2,w_\infty)\in\L^2\times\L^\infty$ be a decompositon of $w$ with disjoint supports and define for $M\in\N$, $w^M=w_\infty+w_2\indic{|w_2|\leqslant M}+M\indic{w_2>M}-M\indic{w_2<-M}$. Then, $w^M\in\L^\infty\subset\mathfrak{W}$, converges to $w$ in $\L^2+\L^\infty$ and verifies 
\begin{equation}\label{Eq:Control[wM2]}
    \left[\left(w^M\right)^2\right]\leqslant \left[w^2\right]
\end{equation}
Let us point out that such a decomposition of $w$, with disjoint supports, is possible as $w$ is bounded at infinity but may fail in general. Even if the condition that $w$ is bounded at infinity appears as a technical one, it may be natural from a physical view point to have an interaction kernel which decays to $0$ at infinity.\\

We now recall the result from \cite{PicklKnowles} we will use in the following.

\begin{prop}[Theorem 3.1. and Remark 3.8. in \cite{PicklKnowles}]\label[prop]{Prop:PicklKnowles}
    
    Let $-A$ be a lower bounded self-adjoint operator on $\L^2$ with form domain $\mathcal{Q}_1$. Let $\mathcal{Q}_N$ be the form domain of $A_{x_1}+\cdots+A_{x_N}$ on $(\L^2)^{\otimes N}$. 
    Let $w\in\mathfrak{W}$. Assume
    $$-H_{N} =-\sum_{j=1}^N A_{x_j} - \frac{1}{N}\sum_{1\leqslant i<j\leqslant N} w(x_j-x_i)$$
    is a lower bounded self-adjoint operator on $(\L^2)^{\otimes N}$ with form domain continuously embedded in $\mathcal{Q}_N$. Let $\Psi_{N,0}\in\mathcal{Q}_N$ such that $|\Psi_{N,0}|_{\L^2}=1$ and $u_0\in\mathcal{Q}_1$ such that $|u_{0}|_{\L^2}=1$. Let $\Psi_N(t)=\e^{\complexI t H_N}\Psi_{N,0}$ and assume there exists a unique maximal solution $(I,u)$ of \cref{Eq:Hartree} in $\mathcal{C}(I,\mathcal{Q}_1)\cap\mathcal{C}^1(I,\mathcal{Q}'_1)$ starting from $u_0$. Let $q(t) = 1-|u(t,x_1)\rangle\langle u(t,x_1)|$ and $\alpha_N(t) = \left|q(t)\Psi_N(t)\right|_{\L^2}^2$. Then, for any $t\in I$, it holds
    $$\alpha_N(t)\leqslant \exp\left(32\left[w^2\right]\int_0^t|u(\tau)|_{\mathcal{Q}_1}^2\d\tau\right)\left(\alpha_N(0)+\frac{1}{N}\right).$$

\end{prop}

\subsection{Some general results on Hartree-type equations}

We first give some general abstract results on Hartree equations we will use later on. In this subsection, $A$ can be any unbounded self-adjoint operator on $\L^2$. In this context, we study some a priori properties of solution of \cref{Eq:Hartree}.

\begin{lemme}\label[lemme]{Lem:MassConservation}

    Let $r_1,r_2\in[1,+\infty]$, $w\in\L^{r_1}+\L^{r_2}$, $I\subset\R$ an interval and $u$ a solution of \cref{Eq:Hartree} in $\L^\infty_{loc}(I,\L^2)\cap\L^3_{loc}\left(I,\L^{\frac{6r_1}{3r_1-2}}\cap\L^{\frac{6r_2}{3r_2-2}}\right)$. Then, the $\L^2$ norm of $u$ is constant on $I$.
    
\end{lemme}

\begin{proof}
    
    Write $w=w_{1}+w_2$ with $w_j\in\L^{r_j}$. Let $P_N = \indic{|A|\leqslant N}$, then for any $f\in\L^2$, $P_N f\in\mathrm{D}(A)$ and converges strongly to $f$ in $\L^2$, and $|P_N|_{\mathcal{L}(\L^2)}=1$. Moreover, it holds
    $$\frac{1}{2}\frac{\d}{\d t}|P_Nu|_{\L^2}^2 = (iA P_N u,P_N u)_{\L^2} + (iw*|u|^2u,P_N u)_{\L^2},$$
    where we use the real Hilbert structure on $\L^2$, given by
    $$(f,g)_{\L^2}=\re\int f\overline{g}\d x.$$
    As $A$ is self-adjoint, it holds for any $s,t\in I$,
    $$|P_Nu(t)|_{\L^2}^2=|P_Nu(s)|_{\L^2}^2+2\int_{s}^{t}(iw*|u|^2u(\tau),P_N u(\tau))_{\L^2}\d\tau.$$
    Now, by Hölder's inequality and Young's inequality for convolution,
    $$\left(iw*|u|^2u(\tau),P_N u(\tau)\right)_{\L^2}\leqslant \left(|w_1|_{\L^{r_1}}|u(\tau)|_{\L^{\frac{6r_1}{3r_1-2}}}^3+|w_2|_{\L^{r_2}}|u(\tau)|_{\L^{\frac{6r_2}{3r_2-2}}}^3\right)|u(\tau)|_{\L^2}$$
    which is locally integrable by hypothesis. We conclude by dominated convergence.

\end{proof}

\begin{lemme}\label[lemme]{Lem:AbstractHartreeUniqueness}

    Let $r_1,r_2\in[1,+\infty]$, $w\in\L^{r_1}+\L^{r_2}$, $I\subset\R$ an interval containing $0$ and $u_0\in\L^2$. There is at most one solution of \cref{Eq:Hartree} in $\L^\infty_{loc}(I,\L^2)\cap\L^2_{loc}\left(I,\L^{\frac{2r_1}{r_1-1}}\cap\L^{\frac{2r_2}{r_2-1}}\right)$ starting from $u_0$.
    
\end{lemme}

\begin{proof}

    Write $w=w_{1}+w_2$ with $w_j\in\L^{r_j}$. Let $u$ and $v$ be two such solutions and $R=u-v$. Then, $R\in\L^\infty_{loc}(I,\L^2)\cap\L^2_{loc}\left(I,\L^{\frac{2r_1}{r_1-1}}\cap\L^{\frac{2r_2}{r_2-1}}\right)$ verifies
    $$\complexI\p_t R + AR +w*|u|^2R+w*\left[(|u|+|v|)(|u|-|v|)\right]v = 0.$$
    Thus, arguing as in the proof of \cref{Lem:MassConservation}, it holds
    $$\frac{1}{2}\frac{\d}{\d t}|R|_{\L^2}^2 = -\im\int |R|^2w*|u|^2 + w*\left[(|u|+|v|)(|u|-|v|)\right]v\overline{R} \d x.$$
    Then, by Hölder's inequality and Young's inequality for convolution,
    $$\frac{1}{2}\frac{\d}{\d t}|R|_{\L^2}^2 \leqslant |w|_{\L^{r_1}+\L^{r_2}}\left( \left(|u|_{\L^{\frac{2r_1}{r1-1}}}+|v|_{\L^{\frac{2r_1}{r_1-1}}} \right)|v|_{\L^{\frac{2r_1}{r_1-1}}}+\left(|u|_{\L^{\frac{2r_2}{r_2-1}}}+|v|_{\L^{\frac{2r_2}{r_2-1}}} \right)|v|_{\L^{\frac{2r_2}{r_2-1}}}\right)|R|_{\L^2}^2.$$
    As $|R(0)|_{\L^2}=0$ and $$\left(|u|_{\L^{\frac{2r_1}{r1-1}}}+|v|_{\L^{\frac{2r_1}{r_1-1}}} \right)|v|_{\L^{\frac{2r_1}{r_1-1}}}+\left(|u|_{\L^{\frac{2r_2}{r_2-1}}}+|v|_{\L^{\frac{2r_2}{r_2-1}}} \right)|v|_{\L^{\frac{2r_2}{r_2-1}}}\in\L^1_{loc}(I),$$ Gronwall's lemma implies $R=0$.

\end{proof}

\begin{rem}By interpolation and \cref{Lem:MassConservation}, any solution of \cref{Eq:Hartree} belonging to $\L^\infty_{loc}(I,\L^2)\cap\L^2_{loc}(I,\L^{\frac{2r_1}{r_1-1}}\cap\L^{\frac{2r_2}{r_2-1}})$ has constant $\L^2$ norm.
\end{rem}

\begin{rem}

    In \cref{Lem:MassConservation,Lem:AbstractHartreeUniqueness}, one can even consider the set $\mathcal{M}$ of signed radon measures instead of $\L^1$, the proofs are the same except $|w_j|_{\L^1}$ should be replaced by $|w_j|_{\mathcal{M}}$. It allows us to consider $w=\lambda\delta_0$, i.e. the pure power case.

\end{rem}

\subsection{An abstract global wellposedness result}

We now give an abstract global wellposedness result on Hartree equations we will use later on. In this subsection, $A$ is an unbounded self-adjoint operator on $\L^2$ satisfying the following assumptions.
\begin{hyp}\label[hyp]{Hyp:Q(A)}{~}

    \begin{enumerate}
        \item $-A$ is lower bounded with form domain $\Q(A)$ and there exists $\delta_A,c_A,C_A>0$ such that
        $$c_A|u|_{\Q(A)}^2\leqslant\langle -Au,u\rangle+\delta_A|u|_{\L^2}^2\leqslant C_A|u|_{\Q(A)}^2.$$
        \item $\Q(A)$ is compactly embedded in $\L^2\cap\L^4$.
    \end{enumerate}

\end{hyp}

\cref{Hyp:Q(A)} 1. deserves a comment. This condition may seem trivial, as any lower bounded operator verifies \cref{Hyp:Q(A)} 1. with $c_A=C_A=1$, if $|u|_{\Q(A)}$ is chosen to be $\langle -Au,u\rangle+\delta_A|u|_{\L^2}^2$. Yet, we formulate our hypothesis like this in order to deal with the cases where $A=A_\theta$ depends on an external parameter in such a way that one can choose $|u|_{\Q(A_\theta)}$ indepently of $\theta$. This will be the case for Anderson operators $H+\xi$, for which $\Q(A)=\rho\Q(H)$ with $\rho$ a bounded function depending continuously on $\xi$, as we will explain below.\\

Our abstract global wellposedness result relies on a compactness argument. First, we deal with the case $w\in\L^\infty$.

\begin{prop}

    Let $d\in\N^*$, $w\in\L^\infty(\R^d)$ and $A:\mathrm{D}(A)\subset\L^2(\R^d)\to\L^2(\R^d)$ be a densely defined unbounded self-adjoint operator. Let $u_0\in\L^2(\R^d)$, there exists a unique maximal $\L^2$-valued solution of \cref{Eq:Hartree} starting from $u_0$, which is global, has constant $\L^2$-norm and is continuous in $\L^\infty([-T,T],\L^2(\R^d))$ with respect to the initial data, for any $T>0$.
    
\end{prop}

\begin{proof}

    The case $w=0$ follows from Stone's theorem \cite{StoneOneParameterGroup}. Now, assume $w\neq 0$. A priori uniqueness is given by \cref{Lem:AbstractHartreeUniqueness}. The result follows by a standard fix point argument in $\L^\infty([-T,T],\L^2(\R^d))$ and the conservation of mass given by \cref{Lem:MassConservation}. By Stone's theorem \cite{StoneOneParameterGroup}, there exists a propagator $U$ associated to $A$. By the Duhamel formula, to find a solution of \cref{Eq:Hartree} in $\L^\infty([-T,T],\L^2(\R^d))$ it is sufficient to solve
    \begin{equation}\label{Eq:DuhamelHartree}
        u(t) = U(t)u_0 - \complexI\int_0^t U(t-s)\left[w*|u(s)|^2u(s)\right]\d s
    \end{equation}
    in the same space. Let $X_T = \L^\infty([-T,T],\L^2(\R^d))$ and for $u\in X_T$,
    $$\Phi(u)(t) = U(t)u_0 - \complexI\int_0^t U(t-s)\left[w*|u(s)|^2u(s)\right]\d s.$$
    Thus, solutions of \cref{Eq:DuhamelHartree} in $X_T$ are exactly the fixed points of $\Phi$. Let $R=2|u_0|_{\L^2}$ and $B_T = \overline{B_{X_T}(R)}$. Now remark that, by unitarity of $U$ and Hölder's and Young's inequalities, for $u,v\in B_T$,
    \begin{equation}
        |\Phi(u)(t)|_{\L^2}\leqslant |u_0|_{\L^2} + T|w|_{\L^\infty}|u|_{X_T}^3\leqslant R\left(\frac{1}{2}+T|w|_{\L^\infty}R^2\right),
    \end{equation}
        and
    \begin{equation}
        |\Phi(u)(t)-\Phi(v)(t)|_{\L^2}\leqslant T|w|_{\L^\infty}(|u|_{X_T}+|v|_{X_T})^2|u-v|_{X_T}\leqslant4R^2 T|w|_{\L^\infty}|u-v|_{X_T}.
    \end{equation}
    Thus, by Banach fix point theorem, for $T<(8|u_0|_{\L^2}^2|w|_{\L^\infty})^{-1}$, $\Phi$ has a unique fixed point in $B_T$, which is the unique solution in $X_T$. Set $T=(9|u_0|_{\L^2}^2|w|_{\L^\infty})^{-1}$, by \cref{Lem:MassConservation}, for all $t\in[-T,T]$, $|u(t)|_{\L^2}=|u_0|_{\L^2}$. Thus, one can construct inductively a sequence of (unique) solutions $u^n$ to \cref{Eq:Hartree} in $X_{nT}$, for any $n\in\N^*$. By \cref{Lem:AbstractHartreeUniqueness}, 
    $$u(t) = \sum_{n=1}^{+\infty} u^n(t)\indic{(n-1)T\leqslant |t|<nT} $$
    is the unique solution of \cref{Eq:Hartree} in $\L^\infty_{loc}(\R,\L^2(\R^d))$. Moreover, it has a constant $\L^2$-norm. Finally, let $u$ and $v$ be the unique solutions of \cref{Eq:Hartree} in $\L^\infty_{loc}(\R,\L^2(\R^d))$ starting respectively at $u_0$ and $v_0$. Let $R> \max(|u_0|_{\L^2},|v_0|_{\L^2})$. We will prove, by induction, there exists $T=T(R)>0$, such that for any $n\in\N^*$,
    $$|u-v|_{X_{nT}}\leqslant 2^n |u_0-v_0|_{\L^2},$$
    thus proving the continuity with respect to the initial data. It holds
    $$u(t)-v(t) = U(t)(u_0-v_0)-\complexI\int_0^t U(t-s)\left[w*|u(s)|^2u(s)-w*|v(s)|^2v(s)\right]\d s,$$
    thus, by conservation of mass,
    $$|u-v|_{X_T}\leqslant |u_0-v_0|_{\L^2}+T|w|_{\L^\infty}\left(|u_0|_{\L^2}+|v_0|_{\L^2}\right)^2|u-v|_{X_T}.$$
    Let $T = \left(8|w|_{\L^\infty}R^2\right)^{-1}$, then it holds
    $$|u-v|_{X_T}\leqslant 2 |u_0-v_0|_{\L^2}.$$
    Now, assume
    $$|u-v|_{X_{nT}}\leqslant 2^n |u_0-v_0|_{\L^2},$$
    then 
    \begin{align*}
        |u-v|_{X_{(n+1)T}} &\leqslant \max\left(|u-v|_{X_{nT}},\sup_{nT\leqslant |t|\leqslant (n+1)T}|u(t)-v(t)|_{\L^2}\right)\\
        &\leqslant \max\left(2^n|u_0-v_0|_{\L^2},2|u(-nT)-v(-nT)|_{\L^2},2|u(nT)-v(nT)|_{\L^2}\right)\\
        &\leqslant \max\left(2^n|u_0-v_0|_{\L^2},2|u-v|_{X_{nt}}\right)\\
        &\leqslant 2^{n+1}|u_0-v_0|_{\L^2}.
    \end{align*}

\end{proof}

We now turn ourselves to solutions of higher regularity. Recall the energy associated to \cref{Eq:Hartree} is defined in \cref{Eq:DefEnergyHartree} and is formally preserved.

\begin{lemme}\label[lemme]{Lem:AbstractHartreeD(A)}

    Let $w\in\L^\infty$, $I\subset\R$ an interval containing $0$ and $u_0\in\mathrm{D}(A)$. Assume $u$ is a solution of \cref{Eq:Hartree} in $\L^\infty_{loc}(I,\L^2)$ starting from $u_0$. Then, $u\in\L^{\infty}_{loc}(I,\mathrm{D}(A))$ and has constant energy.
    
\end{lemme}

\begin{proof}

    Let $u$ be a solution of \cref{Eq:Hartree} in $\L^\infty_{loc}(I,\L^2)$ starting from $u_0$. First, by \cref{Lem:MassConservation} $|u(t)|_{\L^2}=|u_0|_{\L^2}$ for all $t\in I$. Then, it holds
    \begin{align*}
        |u|_{\mathrm{D}(A)}&\leqslant|\p_tu|_{\L^2}+|w*|u|^2u|_{\L^2}+|u_0|_{\L^2}\\
        &\leqslant|\p_tu|_{\L^2} + |w|_{\L^\infty}|u_0|_{\L^2}^3+|u_0|_{\L^2}.
    \end{align*}
    Let $f=\p_tu$, then $f(0)=\complexI\left(Au_0+w*|u_0|^2u_0\right)\in\L^2$ and $f$ solves
    $$\complexI\p_tf+Af+w*|u|^2f+2\re(w*(\overline{u}f))u=0.$$
    We write it in mild form to get
    $$f(t)=\e^{\complexI t A}f(0)-i\int_0^t\e^{\complexI (t-s) A}\left[w*|u|^2f+2\re(w*(\overline{u}f))u\right](s)\d s.$$
    By taking the $\L^2$ norm we get
    $$|f(t)|_{\L^2}\leqslant |f(0)|_{\L^2}+3|w|_{\L^\infty}|u_0|_{\L^2}^2\int_0^t|f(s)|_{\L^2}\d s.$$
    Then, by Gronwall lemma, for all $t\in I$,
    $$|f(t)|_{\L^2}\leqslant |f(0)|_{\L^2}\exp\left(3|t||w|_{\L^\infty}|u_0|_{\L^2}^2\right)$$
    which is locally bounded. Thus, $u\in\L^{\infty}_{loc}(I,\mathrm{D}(A))$. Hence, $\mathcal{E}(u)\in\mathcal{D}'(I)$ and has weak derivative
    \begin{align*}
        \frac{\d}{\d t}\mathcal{E}(u) &=\left(\p_t u,\nabla\mathcal{E}(u)\right)_{\L^2}\\
        &= - \left(\complexI(Au+w*|u|^2u),Au+w*|u|^2u\right)_{\L^2} = 0
    \end{align*}
    as $\p_t u(t),\nabla\mathcal{E}(u(t))\in\L^2$ for almost every $t\in I$. Thus, $\mathcal{E}(u)$ is constant.

\end{proof}

We can now state our general statement.

\begin{cor}\label[cor]{Cor:GWPQ(A)}

    Let $A$ be an unbounded self-adjoint operator on $\L^2$ satisfying \cref{Hyp:Q(A)}. Let $w\in\mathfrak W$ and $u_0\in\Q(A)$. There exists a unique maximal $\Q(A)$-valued solution of \cref{Eq:Hartree} starting from $u_0$. Moreover, it is global and has constant $\L^2$-norm and energy and verifies
    \begin{equation}\label{Eq:ControlNormQ(A)u}
            \left|u\right|_{\L^\infty_t\Q(A)}^2 \lesssim \frac{1}{c_A}\left((1+C_A)|u_0|_{\Q(A)}^2 + \left(1+\frac{1}{c_A}\right) \left[w^2\right]|u_0|_{\L^2}^6+\delta_A|u_0|_{\L^2}^2\right).
    \end{equation}

\end{cor}

\begin{proof}

    We decompose the proof into several steps.
    \begin{itemize}
        \item \textbf{Energy estimate for bounded interactions :} Assume $w\in\L^\infty$. Let $(u_0^n)\subset\mathrm{D}(A)$ be a sequence converging to $u_0$ in $\Q(A)$. Let $u^n$ be the unique solution of \cref{Eq:Hartree} starting from $u_0^n$. Then, $u^n$ converges in $\L^\infty_{loc}(\R,\L^2)$ to $u$. Moreover, by \cref{Lem:AbstractHartreeD(A),Lem:MassConservation}, $u^n\in\L^\infty_{loc}(\R,\mathrm{D}(A))$ and has constant mass and energy. Thus,
        \begin{align*}
            \left|u^n(t)\right|_{\Q(A)}^2&\leqslant \frac{1}{c_A} \langle -Au^n(t),u^n(t)\rangle+\frac{\delta_A}{c_A}|u_0|_{\L^2}^2\\
            &\lesssim \frac{1}{c_A}\left(\mathcal{E}(u^n(t)) + \int \left|u^n(t)\right|^2w*\left|u^n(t)\right|^2\d x+\delta_A|u_0|_{\L^2}^2\right)\\
            &\lesssim \frac{1}{c_A}\left(\mathcal{E}(u^n_0) + |w|_{\L^\infty}|u_0|_{\L^2}^4+\delta_A|u_0|_{\L^2}^2\right).
        \end{align*}
        As $u_0^n$ converges to $u_0$ in $\Q(A)$, $(u^n)_n$ is uniformly bounded in $\L^\infty(\R,\Q(A))$. By Banach-Alaoglu theorem and convergences of $u^n$ to $u$ in  $\L^\infty_{loc}(\R,\L^2)$, $u\in\L^\infty(\R,\Q(A))$ and $u^n(t)$ converges weakly to $u(t)$ in $\Q(A)$ for almost every $t\in\R$. Thus, by sequential weak lower semicontinuity of quadratic forms,
        \begin{align*}
            \frac{1}{2}\left\langle -Au(t),u(t)\right\rangle &\leqslant \liminf_{n\to+\infty} \frac{1}{2}\left\langle -Au^n(t),u^n(t)\right\rangle\\
            &\leqslant \liminf_{n\to+\infty} \mathcal{E}(u_0^n)+\frac{1}{4}\int \left|u^n(t)\right|^2w*\left|u^n(t)\right|^2\d x\\
            &\leqslant \mathcal{E}(u_0)+\frac{1}{4}\int \left|u(t)\right|^2w*\left|u(t)\right|^2\d x.
        \end{align*}
        Thus, $\mathcal{E}(u(t))\leqslant\mathcal{E}(u_0)$ and we conclude $\mathcal{E}(u(t))=\mathcal{E}(u_0)$  by reversing time.
        \item \textbf{Energy estimate for general interactions :} Let $w\in\mathfrak{W}$ and for $M\in\N^*$, 
        $$w^M=w_\infty+w_2\indic{|w_2|\leqslant M}+M\indic{w_2>M}-M\indic{w_2<-M},$$
        with $w=w_2+w_\infty\in\L^2+\L^\infty$ a decomposition of $w$ with disjoint supports. Let $u^M$ be the unique solution of \cref{Eq:Hartree} associated to $w^M$ and starting from $u_0$. By the previous point, $u^M\in\L^\infty(\R,\Q(A))$ and has constant energy $\mathcal{E}_M$. Thus, by \cref{Eq:Control[w]QL2,Eq:Control[wM2],Lem:MassConservation},
        \begin{align*}
            \left|u^M(t)\right|_{\Q(A)}^2 &\leqslant \frac{1}{c_A}\left( \left\langle -Au^M(t),u^M(t)\right\rangle+\delta_A|u_0|_{\L^2}^2\right)\\
            &\lesssim \frac{1}{c_A}\left(\mathcal{E}_M(u_0) + \int \left|u^M(t)\right|^2w^M*\left|u^M(t)\right|^2\d x+\delta_A|u_0|_{\L^2}^2\right)\\
            &\lesssim \frac{1}{c_A}\left(C_A|u_0|_{\Q(A)}^2 + \left[(w^M)^2\right]^{\frac{1}{2}}\left(|u_0|_{\Q(A)}+\left|u^M(t)\right|_{\Q(A)}\right)|u_0|_{\L^2}^3+\delta_A|u_0|_{\L^2}^2\right)\\
            &\lesssim \frac{1}{c_A}\left(C_A|u_0|_{\Q(A)}^2 + \left[w^2\right]^{\frac{1}{2}}\left(|u_0|_{\Q(A)}+\left|u^M(t)\right|_{\Q(A)}\right)|u_0|_{\L^2}^3+\delta_A|u_0|_{\L^2}^2\right). 
        \end{align*}
        By Young's inequality, we get
        \begin{equation}\label{Eq:ControlNormQ(A)uM}
            \left|u^M(t)\right|_{\Q(A)}^2 \lesssim \frac{1}{c_A}\left((1+C_A)|u_0|_{\Q(A)}^2 + \left(1+\frac{1}{c_A}\right) \left[w^2\right]|u_0|_{\L^2}^6+\delta_A|u_0|_{\L^2}^2\right).
        \end{equation}
        Thus, $(u^M)_M$ is uniformly bounded in $\L^\infty(\R,\Q(A))$. Using \cref{Eq:Hartree,Hyp:Q(A)}, $(\p_t u^M)_M$ is uniformly bounded in $\L^\infty(\R,\Q'(A))$. For $T>0$, by compactness of the embedding of $\Q(A)$ in $\L^2\cap\L^4$, using Aubin-Lions-Simon lemma \cite{SimonAubinLions} and Banach-Alaoglu theorem, up to a subsequence, there exists $u\in\L^\infty([-T,T],\Q(A))$ such that $u^M$ converges strongly to $u$ in $\L^\infty([-T,T],\L^2\cap\L^4)$ and for almost every $t\in[-T,T]$, $u^M(t)$ converges weakly to $u(t)$ in $\Q(A)$.
        \item \textbf{Identification of the limit :} By \cref{Lem:AbstractHartreeUniqueness}, there is at most one local $\L^2\cap\L^4$-valued solution of \cref{Eq:Hartree} starting from $u_0$. As $u\in\L^\infty([-T,T],\L^2\cap\L^4)$, it is sufficient to show $u$ verifies \cref{Eq:Hartree}. Let $\phi\in\Q(A)$ and $M\in\N^*$, it holds, for $t\in[-T,T]$,
        \begin{equation}
            \langle \complexI u^M(t),\phi\rangle-\langle \complexI u_0,\phi\rangle = \int_{0}^t \left(\langle -Au^M(s),\phi\rangle - \re\int w^M*\left|u^M(s)\right|^2u^M(s)\overline{\phi}\right)\d s.
        \end{equation}
        As $u^M$ converges uniformly to $u$ in $\L^2\cap\L^4$, $\langle \complexI u^M(t),\phi\rangle$ converges to $\langle \complexI u(t),\phi\rangle$ and
        $$\int w^M*\left|u^M(t)\right|^2u^M(t)\overline{\phi}$$
        converges uniformly to
        $$\int w*\left|u(t)\right|^2u(t)\overline{\phi},$$
        by using the multilinear bound
        $$\left|\int v*(fg)h\phi\right|\leqslant|v|_{\L^2+\L^\infty}|f|_{\L^2\cap\L^4}|g|_{\L^2\cap\L^4} |h|_{\L^2}|\phi|_{\L^2}.$$
        Finally, for almost every $s\in[-T,T]$, $\langle -Au^M(s),\phi\rangle$ converges to $\langle -Au(s),\phi\rangle$ and is bounded by a constant uniform in $M$. Thus, by dominated convergence, it holds
        \begin{equation}
            \langle \complexI u(t),\phi\rangle-\langle \complexI u_0,\phi\rangle = \int_{0}^t \left(\langle -Au(s),\phi\rangle - \re\int w*\left|u(s)\right|^2u(s)\overline{\phi}\right)\d s,
        \end{equation}
        i.e. $u$ is a solution of \cref{Eq:Hartree} in $\L^\infty([-T,T],\Q(A))\subset\L^\infty([-T,T],\L^2\cap\L^4)$. Let $u_N$ be a solution of \cref{Eq:Hartree} in $\L^\infty([-N,N],\Q(A))$. By \cref{Lem:AbstractHartreeUniqueness} it is unique, thus $(u_{N+1})_{|[-N,N]}=u_N$ and
        $$u(t) = \sum_{N=1}^{+\infty} u_N(t)\indic{N-1\leqslant |t|<N}$$
        is the unique maximal $\Q(A)$-valued solution of \cref{Eq:Hartree} starting from $u_0$.
        \item \textbf{Conservation laws of the limit :} By \cref{Lem:MassConservation}, the $\L^2$-norm of $u$ is constant. Now, let $t\in\R$, by sequential weak lower semicontinuity and the convergences obtained in the previous point, it holds
        \begin{align*}
            \frac{1}{2}\langle -Au(t),u(t)\rangle &\leqslant \liminf_{M\to+\infty}\frac{1}{2}\left\langle -Au^M(t),u^M(t)\right\rangle\\
            &\leqslant \liminf_{M\to+\infty} \mathcal{E}_M(u_0) + \frac{1}{4}\int \left|u^M(t)\right|^2w^M*\left|u^M(t)\right|^2\\
            &\leqslant \mathcal{E}(u_0) + \frac{1}{4}\int \left|u(t)\right|^2w*\left|u(t)\right|^2.
        \end{align*}
        Thus, $\mathcal{E}(u(t))\leqslant\mathcal{E}(u_0)$ and we conclude $\mathcal{E}(u(t))=\mathcal{E}(u_0)$ by reversing time. Finally, from \cref{Eq:ControlNormQ(A)uM}, we deduce \cref{Eq:ControlNormQ(A)u}.
     \end{itemize}
\end{proof}

\section{The confining Anderson hamiltonian}

\subsection{Hermite-Sobolev spaces}\label{Sub:Confining_potential}

We denote by $H=\Delta-|x|^2$ the Hermite operator. It is well-known that this operator has eigenfunctions $(h_k)_{k\in\N}$ verifying the relation $-H h_k = \lambda_k^2 h_k$ where $\lambda_k^2\sim c_d k^{\frac{1}{d}}$. Moreover, $(h_k)_{k\in\N}$ is a complete orthonormal system of $\L^2(\mathbb{R}^d,\R)$.\\

Let $s\in\R$ and $p\in(1,+\infty)$ and define
$$\W^{s,p} =\left\{u\in\mathcal{S}'(\mathbb{R}^2), (-H)^{\frac{s}{2}}u\in\L^p(\mathbb{R}^2)\right\}$$
endowed with the norm $$|u|_{\W^{s,p}} = |(-H)^{\frac{s}{2}}u|_{\L^p}.$$
It is known that these spaces are Banach spaces and that for $q$ such that $\frac{1}{p}+\frac{1}{q}=1$, we have $\left(\W^{s,p}\right)'=\W^{-s,q}$ with equal norm. Moreover, the subspace $\vspan\left\{h_k, k\in\N\right\}$ is dense in $\W^{s,q}$-spaces (see for example \cite{BongioanniHSspaces} for the case $s\geqslant0$) and for $s\geqslant0$, an equivalent norm is given for $p\in(1,+\infty)$ by (see \cite{Dziubanski})
$$|u|_{\W^{s,p}}\approx|u|_{\H^{s,p}}+|\langle x\rangle^s u|_{\L^p},$$
where $\H^{s,p}$ are the Bessel-Sobolev spaces with norm
$$|u|_{\H^{s,p}}=\left|\F^{-1}\langle \eta\rangle^{s}\F u\right|_{\L^p_\eta}.$$
This shows $\W^{s,p}=\left\{u\in \H^{s,p}, \langle x\rangle^s u\in\L^p\right\}.$ Using this norm equivalence, we define $\W^{2,p}$ for $p\in\{1,+\infty\}$ as
$$\W^{2,p} = \left\{u\in \L^p,\; \forall \alpha\in\N^d, \left(|\alpha|\leqslant 2 \Rightarrow \p^\alpha u\in\L^p\right) \text{ and } \langle x\rangle^{2} u\in\L^p\right\}$$
endowed with the norm
$$|u|_{\W^{2,p}_x} = \sum_{|\alpha|\leqslant 2}|\p^\alpha u|_{\L^p}+|\langle x\rangle^{2} u|_{\L^p_{x}}$$
and extend this definition to $s\in(0,2)$ by complex interpolation. Moreover inspired by the relation $\left(\W^{s,p}\right)'=\W^{-s,q}$, for $s\in[-2,0)$, we define $\W^{s,\infty} = \left(\W^{-s,1}\right)'$. As usual, we cannot define spaces of negative $\L^1$ regularity as dual spaces of positive $\L^\infty$ regularity.

\begin{rem}{~}
    We do not claim that the first definition of $\W^{s,p}$ agrees with the one we give for spaces over $\L^1$ and $\L^\infty$. We use the same notation for simplicity as we will always control $\W^{s,\infty}$ norms by Sobolev embeddings.
\end{rem}

The choices we made ensure us that we can interpolate between spaces of positive regularity, and give us some other convenient properties that we need in our analysis. For example, our choices allow us to have a simple product rule on our spaces, whose proof follows from Leibniz rule, Hölder inequality and complex bilinear interpolation. We start with product rules, recalling an important estimate going back to Kato and Ponce \cite{KatoPonce}. 
\begin{prop}{(Theorem 1.4 in \cite{GulisashviliKon96})} \label[prop]{Prop:Kato-Ponce}
    Let $r\in(1,+\infty)$ and $s\geqslant 0$. For any $1<p_1,q_1,p_2,q_2\leqslant +\infty$ with $\frac{1}{r}=\frac{1}{p_j}+\frac{1}{q_j}$ ($j\in\{1,2\}$), there exists $C>0$ such that, for any $f\in \mathrm{H}^{s,p_1}\cap\L^{p_2}(\R^2)$ and any $g\in \mathrm{H}^{s,q_2}\cap\L^{q_1}(\R^2)$, it holds
    $$|fg|_{\mathrm{H}^{s,r}}\leqslant C\left(|f|_{\mathrm{H}^{s,p_1}}|g|_{\L^{q_1}}+|f|_{\L^{p_2}}|g|_{\mathrm{H}^{s,q_2}}\right).$$    
\end{prop}

The proof of \cref{Lem:ProductRule,Cor:ProductRuleNegPosReg,Prop:action_d/dx,Sobolev-embeddings} below are given in \cite{Mackowiak_2025}.

\begin{lemme}\label[lemme]{Lem:ProductRule}
    Let $s\geqslant0$ and $1\leqslant p,q,r\leqslant+\infty$ such that $\frac{1}{p}+\frac{1}{q}=\frac{1}{r}$. There exists a constant $C>0$ such that for all $u\in\W^{s,p}$ and $v\in \mathrm{H}^{s,q}$, we have $uv\in\W^{s,r}$ and $|uv|_{\W^{s,r}}\leqslant C |u|_{\W^{s,p}}|v|_{\mathrm{H}^{s,q}}$.   
\end{lemme}

\begin{cor}\label[cor]{Cor:ProductRuleNegPosReg}
    Let $s\geqslant0$, $1<p<+\infty$ and $1\leqslant q,r\leqslant+\infty$ such that $1-\frac{1}{p}+\frac{1}{q}=1-\frac{1}{r}$. There exists a constant $C>0$ such that for all $u\in\W^{-s,r}$ and $v\in \mathrm{H}^{s,q}$, we have $uv\in\W^{-s,p}$ and $|uv|_{\W^{-s,p}}\leqslant C |u|_{\W^{-s,r}}|v|_{\mathrm{H}^{s,q}}$.
\end{cor}

On $\W^{s,q}$ spaces, derivation and multiplication by a power of $\langle x\rangle$ act directly on the regularity exponent.

\begin{prop}(see Corollary 2.5 in \cite{Mackowiak_2025})  \label[prop]{Prop:action_d/dx}
    Let $s\in\R$, $j\in\{1,2\}$ and $q\in(1,+\infty)$, then $\p_j,x_j\in\mathcal{L}(\W^{s,q},\W^{s-1,q})$.
    
\end{prop}   

The following corollary follows from \cref{Prop:action_d/dx} by Stein's complex interpolation method \cite{SteinInterpolation,VoigtAbstractInterpolation}.

\begin{cor}\label[cor]{Cor:action_V}

    Let $\alpha\geqslant 0$, $s\in\R$ and $q\in(1,+\infty)$. Then, $\langle x\rangle^{\alpha}\in\mathcal{L}(\W^{s,q},\W^{s-\alpha,q})$.
    
\end{cor}

The Sobolev-Hermite spaces verify some continuous embeddings as in the classical Sobolev framework. We will also refer to these embeddings as Sobolev embeddings.

\begin{prop}{(Sobolev embeddings)}\label[prop]{Sobolev-embeddings}
    Let $1< p\leqslant q<+\infty$ and $s>\sigma$ such that $\frac{1}{p}-\frac{s}{2}\leqslant\frac{1}{q}-\frac{\sigma}{2}$. Then $\W^{s,p}$ is continuously embedded in $\W^{\sigma, q}$. Moreover if $1<p< q\leqslant+\infty$ and $s>\sigma$ such that $\frac{1}{p}-\frac{s}{2}<\frac{1}{q}-\frac{\sigma}{2}$ and $\sigma\in[-2,2]$ if $q=+\infty$, then $\W^{s,p}$ is compactly embedded in $\W^{\sigma,q}$.
    
\end{prop}

\subsection{Abstract Anderson operator}

We will present a construction of the Anderson operator $H+\xi$ on $\R^d$ ($d\in\{1,2\}$) through its quadratic form, inspired by \cite{MackowiakStrichartzConfAnderson} (see also \cite{mouzard2023simple} for an analogous construction on the torus). Namely, we will show there exists $Y\in\W^{2-\frac{d}{2}-,\infty}$ and $Z\in\W^{0-,\infty}$ such that the form domain of $H+\xi$ is given by $\e^Y\W^{1,2}$ and
$$\langle(H+\xi)u_1,u_2\rangle = \re\int_{\R^d} \nabla v_1\cdot\overline{\nabla v_2}\e^{2Y}+|x|^2(1-Y)u_1\overline{u_2}\d x - \langle Zu_1,u_2\rangle$$
for $v_j=\e^{-Y}u_j\in\W^{1,2}$. The enhanced noise parameter $\theta=(\xi,Y,Z)$ is then a continuous function of the enhanced noise $\Xi$ constructed in \cite{MackowiakStrichartzConfAnderson} in dimension 2. This abstract setting allows to treat both dimension 1 and 2 at the same time for most of the argument. In dimension 3, this abstract construction is inspired by \cite{mouzard2023simple} and the formal renormalization procedure detailed above. It provides a class of formal renormalizations of the operator $H+\xi$. Proving that one can in fact construct stochastic objects in order to treat the 3d Anderson-Hermite operator will be the object of future work.

We first give a heuristic before giving the rigorous construction. We start with the Sobolev regularity of the spatial white noise. Recall that a white noise $\xi$ is a real-valued centered Gaussian field over $\mathcal{S}(\R^d)$ verifying
$$\forall \phi,\psi\in\mathcal{S}(\R^d)\; \left[\langle \xi,\phi\rangle_{\mathcal{S}',\mathcal{S}}\langle \xi,\psi\rangle_{\mathcal{S}',\mathcal{S}}\right]=\langle\phi,\psi\rangle_{\L^2}.$$
Equivalently, one can define $\xi$ as the random series 
$$\xi=\sum_{k\in\N}\xi_k h_k,$$
where $(\xi_k)_{k\in\N}$ are i.i.d. real-valued standard Gaussian variables and the series converges almost surely in $\mathcal{S}'(\R^d)$.

\begin{lemme}[Corollary 2.3.6. in \cite{PhDMackowiak}]
    Let $d\in\N^*$. Almost surely, for any $q\in[2,+\infty]$ and $s>\frac{d}{q}>\sigma$, $\xi\in\W^{-\frac{d}{2}-s,q}$ and $\xi\notin\W^{-\frac{d}{2}-\sigma,q}$.
\end{lemme}

As $\xi$ is of regularity $-\frac{d}{2}-$, it is too singular to construct directly $H+\xi$ in dimension $d\geqslant 2$. We thus make use of an exponential transform inspired by the one introduced in \cite{Haire_Labbe}. Let $ Y_1=(-H)^{-1}\xi\in\W^{2-\frac{d}{2}-,\infty}$. For $u=\e^{Y_1}v$, it holds formally
\begin{equation}\label{Eq:AndersonOpRenorm0}
    (H+\xi)u=\e^{Y_1}\left(H v + 2\nabla Y_1\cdot\nabla{v}+xY_1\cdot xv+|\nabla Y_1|^2v\right).
\end{equation}
If $d=1$, then $Y_1\in\W^{\frac{3}{2}-,\infty}$, thus one can take $Y=Y_1$ and $Z = |\nabla Y|^2\in\W^{\frac{1}{2}-,\infty}$ and one can define $H+\xi$ through the quadratic form associated to \cref{Eq:AndersonOpRenorm0}. Unfortunately, for $d\geqslant 2$, $Y$ is of regularity $2-\frac{d}{2}-<1$, thus $\nabla Y$ is of negative regularity and thus $|\nabla Y|^2$ is ill-defined. It is well-known that $|\nabla Y_1|^2$ needs to be renormalized (see for example Proposition 3.11. in \cite{Mackowiak_2025} for the 2d case). From now on, assume we have constructed a renormalization $ Z_1\in\W^{2-d-,\infty}$ (as a product of quantities of regularity $1-\frac{d}{2}-<0$) of $|\nabla Y_1|^2$. Then, one can replace $|\nabla Y_1|^2$ by $Z_1$ in \cref{Eq:AndersonOpRenorm0} and get
\begin{equation}\label{Eq:AndersonOpRenorm1}
    (H+\xi)u=\e^{ Y_1}\left(H v + 2\nabla Y_1\cdot\nabla{v} +xY_1\cdot xv+ Z_1v\right)
\end{equation}
Then, \cref{Eq:AndersonOpRenorm1} is sufficient to define $H+\xi$ through its quadratic form in dimension 2.\\

In dimension 3, $ Z_1$ is of regularity $-1-$, as $\xi$ is in dimension 2. Let $ Y_2=(-H)^{-1}Z_1\in\W^{1-,\infty}$, for $u=\e^{ Y_1+ Y_2}v$, it holds formally
\begin{equation}\label{Eq:DefAndersonOp3d1ereRenorm}
    (H+\xi)u=\e^{ Y_1+ Y_2}\left(H v + 2\nabla( Y_1+ Y_2)\cdot\nabla{v} + x(Y_1+Y_2)\cdot xv +(\nabla Y_1\cdot\nabla Y_2)v+ |\nabla  Y_2|^2v\right).
\end{equation}
As in dimension 2, $|\nabla  Y_2|^2$ is ill-defined and should be replaced by a renormalized product $ Z_2$ which formally belongs to $\W^{0-,\infty}$. Likewise, $\nabla Y_1\cdot\nabla Y_2$ is ill-defined as $\nabla Y_1$ is of regularity $-\frac{1}{2}-$ and $\nabla Y_2$ is of regularity $0-$. Thus it has to be replaced by a renormalized product $ Z_{12}$ which formally belongs to $\W^{-\frac{1}{2}-,\infty}$. Thus, \cref{Eq:DefAndersonOp3d1ereRenorm} becomes
\begin{equation}\label{Eq:AndersonOpRenorm2}
    (H+\xi)u=\e^{ Y_1+ Y_2}\left(H v + 2\nabla( Y_1+ Y_2)\cdot\nabla{v}+ x(Y_1+Y_2)\cdot xv +(Z_2+2Z_{12})v\right.
\end{equation}
As $Z_{12}$ is of regularity $-\frac{1}{2}-$ and $\e^{ Y_1+ Y_2}$ is of regularity $\frac{1}{2}-$, even for smooth $v$, $Z_{12}v\e^{ Y_1+ Y_2}$ is ill-defined. Thus, let $Y_3=2(-H)^{-1}Z_{12}$, then for $u=\e^{Y_1+Y_2+Y_3}v$, \cref{Eq:AndersonOpRenorm2} becomes
\begin{equation}\label{Eq:AndersonOpRenorm3}
    (H+\xi)u=\e^{ Y_1+ Y_2+Y_3}\left(H v + 2\nabla( Y_1+ Y_2+Y_3)\cdot\nabla{v}+ x(Y_1+Y_2+Y_3)\cdot xv+ \left(|\nabla Y_3|^2+Z_2\right)v\right),
\end{equation}
which is well-defined as $Y_3\in\W^{\frac{3}{2}-,\infty}$, thus $|\nabla Y_3|^2\in\W^{\frac{1}{2}-,\infty}$. Hence, one can chose $Y=Y_1+Y_2+Y_3$ and $Z=|\nabla Y_3|^2+Z_2$ to define $H+\xi$ through the quadratic form associated to \cref{Eq:AndersonOpRenorm0}.\\

Thus, in dimension $d\in\{1,2\}$, $H+\xi$ can be written as
\begin{equation}\label{Eq:AndersonAbstract}
    (H+\xi)u=\e^{ Y}\left(H v + 2\nabla Y\cdot\nabla{v} + xY\cdot xv+ Zv\right),
\end{equation}
for $u=\e^Y v$, with $Y\in\W^{2-\frac{d}{2}-,\infty}$ and $Z\in\W^{0-,\infty}$. Formally, \cref{Eq:AndersonAbstract} also holds in dimension 3. Let
\begin{equation}\label{Def:ThetaAnderson}
    \Theta=\left\{(\xi,Y, Z)\in\W^{-\frac{d}{2}-,\infty}\times\W^{2-\frac{d}{2}-,\infty}\times\W^{0-,\infty}\right\}.
\end{equation}
Remark that $\Theta$ is a Fréchet space when endowed with the family of seminorms $(N_p)_{p\in\N^*}$ given by
$$N_p(\theta) = |\xi|_{\W^{-\frac{d}{2}-\frac{1}{p},\infty}}+|Y|_{\W^{2-\frac{d}{2}-\frac{1}{p},\infty}}+| Z|_{\W^{-\frac{1}{p},\infty}}.$$
Let us point out $\xi\in\W^{-\frac{d}{2}-,\infty}\mapsto Y_1\in\W^{2-\frac{d}{2}-,\infty}$ is continuous. Thus, in dimension 1 and 2, $\Theta$ can be reduced to the space of enhanced noises
\begin{equation}\label{Def:EnhancedNoiseSpace}
    \mathcal{X}=\left\{(\xi, Z)\in\W^{-\frac{d}{2}-,\infty}\times\W^{0-,\infty}\right\}.
\end{equation}
Moreover, in dimension 1, $Y\in\W^{\frac{3}{2}-,\infty}\mapsto Z=|Y'|^2\in\W^{\frac{1}{2}-,\infty}$ is continuous, thus $\Theta$ can even be reduced to $\W^{-\frac{1}{2}-,\infty}$. Nevertheless, we keep the full parameter $\theta=(\xi,Y,Z)$ to make clearer the dependence on $\theta$ in our estimates.\\

For $\theta\in\Theta$, define $\rho=\e^Y$ and the bilinear form
\begin{equation}\label{Def:AndersonBilinearForm}
    \forall u_1,u_2\in\rho\W^{1,2},\; a_\theta(u_1,u_2) = \re\int \nabla v_1\cdot\overline{\nabla v_2}\rho^2+|x|^2(1-Y)u_1\overline{u_2}\d x - \langle Zu_1,u_2\rangle,
\end{equation}
where $v_j=\rho^{-1}u_j\in\W^{1,2}$. For $s\in[0,1]$, let $\mathcal{D}^{s,2}=\rho\W^{s,2}$ endowed with the natural norm
$$\left|u\right|_{\mathcal{D}^{s,2}}=\left|\rho^{-1}u\right|_{\W^{s,2}}.$$
Then, by complex interpolation applied on $\rho$ and $\rho^{-1}$ seen as multiplication operators, it holds $\left(\mathcal{D}^{1,2},\mathcal{D}^{0,2}\right)_s=\mathcal{D}^{s,2}$ with equal norms,
where $(\cdot,\cdot)_{\cdot}$ is the complex interpolation functor. Moreover, $\mathcal{D}^{0,2}=\L^2$ and 
$$(\inf \rho)|u|_{\L^2}\leqslant|u|_{\mathcal{D}^{0,2}}\leqslant (\sup \rho)|u|_{\L^2}.$$
Then, $(a_\theta,\mathcal{D}^{1,2})$ is a continuous bilinear form as
\begin{equation}\label{Eq:GenContinuityEstimate1particle}
    |a_\theta(u_1,u_2)|\leqslant \left(\left(1+|Y|_{\L^\infty}\right)|\rho|_{\L^\infty}^2 + C|Z|_{\W^{-\frac{1}{4},\infty}}|\rho|_{\H^{\frac{1}{4},\infty}}^2\right)|u_1|_{\mathcal{D}^{1,2}}|u_2|_{\mathcal{D}^{1,2}}
\end{equation}
for some universal constant $C$, by \cref{Lem:ProductRule,Cor:ProductRuleNegPosReg}. Moreover, it quasi-coercive in the following sense.

\begin{lemme}\label[lemme]{Lem:GenQuasiCoercivEstimate1particle}

    There exists a $c>0$ such that for any $\theta\in\Theta$, setting $\delta_\theta=c \e^{32|Y|_{\L^\infty}}\left(|Y|_{\W^{\frac{1}{4},\infty}}^{8}+|Z|_{\W^{-\frac{1}{4},\infty}}^2|\rho|_{\H^{\frac{1}{4},\infty}}^4\right)$, it holds
    $$\forall u\in\mathcal{D}^{1,2},\;  a_\theta(u,u)+\delta_\theta|u|_{\L^2}^2\geqslant \frac{(\inf \rho)^2}{2}|u|_{\mathcal{D}^{1,2}}^2.$$
    
\end{lemme}

\begin{proof}

    Let $u=\rho v\in\rho\W^{1,2}$, by duality, \cref{Lem:ProductRule,Cor:ProductRuleNegPosReg,Cor:action_V}, Sobolev embeddings and interpolation, it holds for some absolute constants $C,C'$ varying from line to line,
    \begin{align*}
        a_\theta(u,u)&\geqslant (\inf \rho)^2|u|_{\mathcal{D}^{1,2}}^2 - C(\sup \rho)^2|Y|_{\W^{\frac{1}{4},\infty}}|u|_{\mathcal{D}^{\frac{7}{8},2}}^2 - C'|Z|_{\W^{-\frac{1}{4},\infty}}|\rho|_{\H^{\frac{1}{4},\infty}}^2|u|_{\mathcal{D}^{\frac{1}{2},2}}^2\\
        &\geqslant (\inf \rho)^2|u|_{\mathcal{D}^{1,2}}^2 - C(\sup \rho)^{\frac{9}{4}}|Y|_{\W^{\frac{1}{4},\infty}}|u|_{\mathcal{D}^{1,2}}^{\frac{7}{4}}|u|_{\L^2}^{\frac{1}{4}} - C'|Z|_{\W^{-\frac{1}{4},\infty}}(\sup\rho)|\rho|_{\H^{\frac{1}{4},\infty}}^2|u|_{\mathcal{D}^{1,2}}|u|_{\L^2}\\
        &\geqslant \frac{(\inf \rho)^2}{2}|u|_{\mathcal{D}^{1,2}}^2 - C\frac{(\sup \rho)^{18}}{(\inf \rho)^{14}}|Y|_{\W^{\frac{1}{4},\infty}}^{8}|u|_{\L^2}^2 - C'\frac{(\sup\rho)^2}{(\inf \rho)^2}|Z|_{\W^{-\frac{1}{4},\infty}}^2|\rho|_{\H^{\frac{1}{4},\infty}}^4|u|_{\L^2}^2\\
        &\geqslant \frac{(\inf \rho)^2}{2}|u|_{\mathcal{D}^{1,2}}^2- C\e^{32|Y|_{\L^\infty}}\left(|Y|_{\W^{\frac{1}{4},\infty}}^{8}+|Z|_{\W^{-\frac{1}{4},\infty}}^2|\rho|_{\H^{\frac{1}{4},\infty}}^4\right)|u|_{\L^2}^2
    \end{align*}
    
\end{proof}

Hence, Theorem 6.2.6 in \cite{KatoPerturbationBook} implies there exists a unique self-adjoint operator $-A_\theta$ such that $\left\langle-A_\theta u_1,u_2\right\rangle = a_\theta(u_1,u_2)$ for all $u_1,u_2\in\Q\left(A_\theta\right)=\mathcal{D}^{1,2}$.

\subsection{The Anderson-Hartree equation}

In this section, $-A=-A_\theta$ for some $\theta\in\Theta$, with $\Q(A)=\mathcal{D}^{1,2}=\rho\W^{1,2}$. By Sobolev embeddings in dimension $d\leqslant 3$, as $\rho\in\L^\infty$, $\mathcal{D}^{1,2}$ is compactly embedded in $\L^2\cap\L^4$. Moreover, \cref{Lem:GenQuasiCoercivEstimate1particle,Eq:GenContinuityEstimate1particle} show one can take $$c_{A} = c_\theta=\frac{(\inf \rho)^2}{2},$$ $$C_A = C_\theta = \left(1+|Y|_{\L^\infty}\right)|\rho|_{\L^\infty}^2 + C|Z|_{\W^{-\frac{1}{4},\infty}}|\rho|_{\H^{\frac{1}{4},\infty}}^2$$ and  $$\delta_A=\delta_\theta =C\e^{32|Y|_{\L^\infty}}\left(|Y|_{\W^{\frac{1}{4},\infty}}^{8}+|Z|_{\W^{-\frac{1}{4},\infty}}^2|\rho|_{\H^{\frac{1}{4},\infty}}^4\right)$$ in \cref{Hyp:Q(A)}, for $C$ large enough and independent of $\theta$. In this context, \cref{Eq:Hartree} becomes \cref{Eq:AndersonHartree}. However, as $\Q(A_\theta)$ depends on $\theta$, the definition of the bracket $[w]$ depends on $\theta$. Yet, remark that
\begin{equation}\label{Eq:Equivalence[w2]theta}
    \e^{-2|Y|_{\L^\infty}}\sup_{v\in\W^{1,2}}\frac{\left|w*|v|^2\right|_{\L^\infty}}{|v|_{\W^{1,2}}^2}\leqslant\sup_{u\in\mathcal{D}^{1,2}}\frac{\left|w*|u|^2\right|_{\L^\infty}}{|u|_{\mathcal{D}^{1,2}}^2}\leqslant\e^{2|Y|_{\L^\infty}}\sup_{v\in\W^{1,2}}\frac{\left|w*|v|^2\right|_{\L^\infty}}{|v|_{\W^{1,2}}^2}.
\end{equation}
Hence, $\mathfrak{W}$ does in fact not depend on $\theta$.\\

Let us point out that, as a consequence of Sobolev embeddings, we can exhibit a lot of elements of $\mathfrak{W}$. For example, in dimension 1, $\mathcal{D}^{1,2}=\W^{1,2}\subset\L^2\cap\L^\infty$ thus for any  $w\in\L^2+\L^\infty$, $\left[w^2\right]<+\infty$. Similarly, in dimension 2, $\mathcal{D}^{1,2}\subset\L^2\cap\L^p$ for $p\in[2,+\infty)$, thus any  $w\in\L^{2+}+\L^\infty$ verifies $\left[w^2\right]<+\infty$. Moreover, for $\eps>0$ and $q\in[2,+\infty)$, still in dimension 2, it holds
$$\left|w^2*|u|^2\right|_{\H^{\eps,q}} \leqslant C(\eps)|w|_{\H^\eps}^2|u|_{\L^{2q}}^{2}.$$
Thus, for any $w\in\H^{0+}+\L^\infty$, $\left[w^2\right]<+\infty$.\\

In dimension 3, Sobolev embeddings are more restricted and $\mathcal{D}^{1,2}\subset\L^2\cap\L^6$. Using purely Sobolev embeddings, we see $\left[w^2\right]<+\infty$ for any $w\in\L^3+\L^\infty$. Moreover, for $p\in[2,3)$, $\eps>0$ and $q\in[2,+\infty)$, it holds
$$\left|w^2*|u|^2\right|_{\H^{\eps,q}} \leqslant C(\eps,p)|w|_{\H^{\eps,p}}^2|u|_{\L^{2r}}^{2}$$
with $\frac{2}{p}+\frac{2}{r}=1+\frac{1}{q}$. Thus $\left[w^2\right]<+\infty$ for any $w\in\H^{\frac{6}{p}-2+,p}$. One can go one step further using Hardy's inequality
\begin{equation}\label{Eq:HardyInequality}
    \forall 1\leqslant p<3,\; \forall f\in\Dot{\H}^{1,p},\; \left|\frac{f}{|x|}\right|_{\L^p}\leqslant C(p)|\nabla f|_{\L^p}.
\end{equation}
Let $p\in[4,+\infty]$, $\lambda\in\L^p(\R^3)$ and $\alpha\in[0,1-\frac{3}{p}]$. Then, using \cref{Eq:HardyInequality}, one can show $w=\lambda|x|^{-\alpha}$ verifies $\left[w^2\right]<+\infty$. Using the fact that $\mathfrak{W}$ is a vector space, the three criteria above give us a wide family of interaction kernels belonging to $\mathfrak{W}$ in dimension 3. In particular, Coulomb's potential $|x|^{-1}$ belongs to $\mathfrak{W}$.\\

In order to apply \cref{Prop:PicklKnowles} and prove \cref{Th:MeanFieldLimitAndersonHartree}, we now check that one can bound the $\mathcal{D}^{1,2}$-norm of solutions locally uniformly in $\theta$. Remark that for $u_0^\theta=\rho v_0$ with a fixed $v_0\in\W^{1,2}$, the unique global solution $u^{\theta,w}$ of \cref{Eq:AndersonHartree} starting from $u_0^\theta$, given by \cref{Cor:GWPQ(A)}, verifies
$$\left|u^{\theta,w}\right|_{\L^\infty_t\mathcal{D}^{1,2}_x}^2\leqslant C\frac{1}{c_\theta}\left((1+C_\theta)|v_0|_{\W^{1,2}}^2 + \left(1+\frac{1}{c_\theta}\right) \left\{w^2\right\}\e^{8|Y|_{\L^\infty}}|v_0|_{\L^2}^6+\delta_\theta\e^{2|Y|_{\L^\infty}}|v_0|_{\L^2}^2\right)$$
by \cref{Eq:ControlNormQ(A)u}, with 
$$\{w\}=\sup_{v\in\W^{1,2}}\frac{\left|w*|v|^2\right|_{\L^\infty}}{|v|_{\W^{1,2}}^2}.$$
the $\theta$-independent bracket. Remark that $\frac{1}{c_\theta}$, $C_\theta$ and $\delta_\theta$ are bounded over bounded sets of $\Theta$. Thus, for any bounded set $B\subset\Theta\times\mathfrak{W}$, there exists $C(B)>0$ such that
\begin{equation}\label{Eq:ControlNormD12theta}
    \sup_{(\theta,w)\in B}\left|u^{\theta,w}\right|_{\L^\infty_t\mathcal{D}^{1,2}_x}^2\leqslant C(B)\left(|v_0|_{\W^{1,2}}^2+|v_0|_{\L^2}^6\right).
\end{equation}

\section{The mean-field limit}

This section is devoted to the proof of \cref{Th:MeanFieldLimitAndersonHartree}.

\subsection{The many-body hamiltonian}

Let $U=\R^d$ with $d\leqslant 3$. Let $w\in\mathfrak{W}$ and $N\geqslant 2$. We now show that $$\mathcal{H}_{N,\theta,w}=\sum_{j=1}^N (A_\theta)_{x_j} + \frac{1}{N}\sum_{1\leqslant i<j\leqslant N}w(x_i-x_j)$$ verifies assumptions of \cref{Prop:PicklKnowles} for any $\theta\in\Theta$. Its associated bilinear form is given for $u_1,u_2\in\rho^{\otimes N}\W^{1,2}(U^N)$ by
\begin{equation}
    q_{N,\theta}(u_1,u_2) = \sum_{j=1}^N \re\int_{U^{N-1}} a_\theta(u_1(\hat{x}_j),u_2(\hat{x}_j))\d\hat{x}_j+ \frac{1}{N}\sum_{1\leqslant i<j\leqslant N}\re\int_{U^N}w(x_i-x_j)u_1(x)\overline{u_2(x)}\d x,
\end{equation}
where $\d\hat{x}_j = \bigotimes_{i\neq j}\d x_i$ and for $f:\R^{Nd}\to\C$, $f(\hat{x}_j):\R^{(N-1)d}\to\C$ is defined such that $f(\hat{x}_j)(x_j)=f(x)$. Let $$\|a_\theta\| = \sup_{\Atop{u_1,u_2\in\mathcal{D}^{1,2}}{u_1,u_2\neq 0}}\frac{|a_\theta(u_1,u_2)|}{|u_1|_{\mathcal{D}^{1,2}}|u_2|_{\mathcal{D}^{1,2}}}<+\infty,$$
then, it holds
\begin{align*}
    |q_{N,\theta}(u_1,u_2)| &\leqslant \|a_\theta\|\sum_{j=1}^N \int_{U^{N-1}} |u_1(\hat{x}_j)|_{\mathcal{D}^{1,2}_{x_j}}|u_2(\hat{x}_j)|_{\mathcal{D}^{1,2}_{x_j}}\d\hat{x}_j \\
    &+ \frac{1}{N}\sum_{1\leqslant i<j\leqslant N}\int_{U^{N-1}}|w*_{x_j}(u_1(\hat{x}_j)u_2(\hat{x}_j))|_{\L^\infty_{x_j}} \d\hat{x}_j.
\end{align*}

Now, by Cauchy-Schwarz,
\begin{align*}
    \sum_{j=1}^N\int_{U^{N-1}} |u_1(\hat{x}_j)|_{\mathcal{D}^{1,2}_{x_j}}|u_2(\hat{x}_j)|_{\mathcal{D}^{1,2}_{x_j}}\d\hat{x}_j &\leqslant \left(\sum_{j=1}^N\int_{U^N}|\nabla_{x_j}v_1|^2+|x_jv_1|^2\d x\right)^{\frac{1}{2}}\left(\sum_{j=1}^N\int_{U^N}|\nabla_{x_j}v_2|^2+|x_jv_2|^2\d x\right)^{\frac{1}{2}}\\
    &=|u_1|_{\mathcal{D}^{1,2}}|u_2|_{\mathcal{D}^{1,2}}.
\end{align*}
Furthermore, by Young's inequality and \cref{Eq:Def[w]}, for any $\eta>0$, it holds
\begin{equation}\label{Eq:ControlWmanybody}
    \begin{split}
        \int_{U^{N-1}}|w*_{x_j}(u_1(\hat{x}_j)u_2(\hat{x}_j))|_{\L^\infty_{x_j}} \d\hat{x}_j &\leqslant \frac{\eta\left[w^2\right]}{2}\int_{U^{N-1}} |u_1(\hat{x}_j)|_{\mathcal{D}^{1,2}_{x_j}}|u_2(\hat{x}_j)|_{\mathcal{D}^{1,2}_{x_j}}\d\hat{x}_j \\
        &+ \frac{1}{2\eta}|u_1|_{\L^2}|u_2|_{\L^2}.
    \end{split}
\end{equation}
Thus,
\begin{align*}
    |q_{N,\theta}(u_1,u_2)| &\leqslant \|a_\theta\||u_1|_{\mathcal{D}^{1,2}}|u_2|_{\mathcal{D}^{1,2}} + \frac{\left[w^2\right]}{2}|u_1|_{\mathcal{D}^{1,2}}|u_2|_{\mathcal{D}^{1,2}}+\frac{N}{2}|u_1|_{\L^{2}}|u_2|_{\L^2}\\
    &\leqslant \left(\|a_\theta\|+\frac{\left[w^2\right]}{2}+\frac{N}{2}(\sup \rho)^2\right)|u_1|_{\mathcal{D}^{1,2}}|u_2|_{\mathcal{D}^{1,2}}.
\end{align*}
Moreover, \cref{Lem:GenQuasiCoercivEstimate1particle,Eq:ControlWmanybody} imply
\begin{align*}
    q_{N,\theta}(u,u)+N\delta_{\theta}|u|_{\L^2}^2&\geqslant \frac{(\inf \rho)^2}{2}\sum_{j=1}^N\int |u(\hat{x}_j)|_{\mathcal{D}^{1,2}}^2\d\hat{x}_j -\frac{1}{N}\sum_{1\leqslant i<j\leqslant N}\int_{U^{N-1}}|w*_{x_j}(u_1(\hat{x}_j)u_2(\hat{x}_j))|_{\L^\infty_{x_j}} \d\hat{x}_j\\
    &\geqslant \frac{(\inf \rho)^2}{2}|u|_{\mathcal{D}^{1,2}}^2  - \frac{\eta\left[w^2\right]}{2}\sum_{j=1}^N\int_{U^{N-1}} |u(\hat{x}_j)|_{\mathcal{D}^{1,2}}^2\d\hat{x}_j - \frac{N-1}{4\eta}|u|_{\L^2}^2\\
    &\geqslant \frac{(\inf \rho)^2}{2}|u|_{\mathcal{D}^{1,2}}^2  - \frac{\eta\left[w^2\right]}{2}|u|_{\mathcal{D}^{1,2}}^2 - \frac{N-1}{4\eta}|u|_{\L^2}^2.
\end{align*}
Choosing $\eta = \frac{(\inf \rho)^2}{2\left[w^2\right]}$ leads to
\begin{equation}\label{Eq:GenManyBodyQuasiCoercivEstimate}
    \forall u\in\rho^{\otimes N}\W^{1,2}(U^N),\; \frac{(\inf \rho)^{2}}{4}|u|_{\mathcal{D}^{1,2}}^2 \leqslant q_{N,\theta}(u,u)+\delta_{N,\theta,w}|u|_{\L^2}^2 \leqslant C_{N,\theta,w} |u|_{\mathcal{D}^{1,2}}^2
\end{equation}
with 
\begin{equation}\label{Eq:deltaNtheta}
    \delta_{N,\theta,w} = N\left(\delta_\theta + \frac{\left[w^2\right]}{2(\inf \rho)^2}\right)
\end{equation}
and 
\begin{equation}\label{Eq:CNtheta}
    C_{N,\theta,w} = \|a_\theta\|+\frac{\left[w^2\right]}{2}+\frac{N}{2}(\sup \rho)^2 + \delta_{N,\theta,w}.
\end{equation}
Let us point out that $\delta_{N,\theta,w}$ and $C_{N,\theta,w}$ are bounded on bounded sets of $\Theta\times\mathfrak{W}$. Moreover, the symmetry of $w$ and \cref{Eq:GenManyBodyQuasiCoercivEstimate} imply $-\mathcal{H}_{N,\theta,w}$ is a lower bounded self-adjoint operator on $(\L^2)^{\otimes N}$ of form domain $\rho^{\otimes N}\W^{1,2}(U^N)$ by Theorem 6.2.6 in \cite{KatoPerturbationBook}. Remark that $\rho^{\otimes N}\W^{1,2}(U^N)$ is also the form domain of $A_{\theta,x_1}+\cdots+A_{\theta,x_N}$, by taking $w=0$. Thus, $-\mathcal{H}_{N,\theta,w}$ satisfies the hypothesis of \cref{Prop:PicklKnowles}.

\subsection{The many-body Schrödinger equation}

Let $N\geqslant 2$, $\Psi_{N,\theta}^0=\rho^{\otimes N}\Phi_N^0\in\rho^{\otimes N}\W^{1,2}(U^N)$ and $\Psi_{N,\theta,w}(t) = \e^{\complexI t \mathcal{H}_{N,\theta,w}}\Psi_{N,\theta}^0$ be the solution of \cref{Eq:ManyBodySchrodinger} starting from $\Psi_{N,\theta}^0$. Finally, let $\Phi_{N,\theta,w}(t)=\left(\rho^{-1}\right)^{\otimes N}\Psi_{N,\theta,w}(t)$. Then, by \cref{Eq:GenManyBodyQuasiCoercivEstimate} and Stone's theorem \cite{StoneOneParameterGroup}, it holds
\begin{equation}\label{Eq:BoundPhiNtheta}
    |\Phi_{N,\theta,w}(t)|_{\W^{1,2}}^2 \leqslant \frac{4C_{N,\theta,w}}{(\inf \rho)^2}|\Phi_N^0|_{\W^{1,2}}^2.
\end{equation}
Remark that $C_{N,\theta,w}$ is bounded on bounded sets of $\Theta\times\mathfrak{W}$. Thus, we prove the following lemma.

\begin{lemme}\label[lemme]{Lem:ContinuityPsiNtheta}

    Let $N\geqslant 2$, $\Psi_{N,\theta}^0=\rho^{\otimes N}\Phi_N^0\in\rho^{\otimes N}\W^{1,2}(U^N)$ and $\Psi_{N,\theta,w}(t) = \e^{\complexI t \mathcal{H}_{N,\theta,w}}\Psi_{N,\theta}^0$. For any $T>0$, the map 
    $$(\theta,w)\in\Theta\times\mathfrak{W}\mapsto \Psi_{N,\theta,w}\in\mathcal{C}\left([-T,T],(\L^2)^{\otimes N}\right)$$
    is continuous.
    
\end{lemme}

\begin{proof}

    Let $\theta,\theta_\eps\in\Theta$ with $\theta_\eps$ converging to $\theta$ and $w,w_\eps\in\mathfrak{W}$ with $w_\eps$ converging to $w$. Then, by \cref{Eq:BoundPhiNtheta}, $(\Phi_{N,\theta_\eps,w_\eps})_{\eps}$ is bounded in $\L^\infty(\R,\W^{1,2}(U^N))$. Using \cref{Eq:GenManyBodyQuasiCoercivEstimate,Eq:BoundPhiNtheta,Eq:ManyBodySchrodinger}, it follows $(\p_t\Phi_{N,\theta_\eps,w_\eps})_{\eps}$ is bounded in $\L^\infty(\R,\W^{-1,2}(U^N))$. Thus, for any $T>0$, by Aubin-Lions-Simon lemma \cite{SimonAubinLions} and Banach-Alaoglu theorem, there exists $\Phi_N\in\L^\infty([-T,T],\W^{1,2}(U^N))$ such that, up to a subsequence, $\Phi_{N,\theta_\eps,w_\eps}$ converges to $\Phi_N$ strongly in $\L^\infty([-T,T],\W^{s,2}(U^N))$ ($s<1$) and $\Phi_{N,\theta_\eps,w_\eps}(t)$ converges to $\Phi_N(t)$ weakly in $\W^{1,2}(U^N)$ for almost every $t\in[-T,T]$. For any $\Phi\in\W^{1,2}(U^N)$, it holds
    \begin{align*}
        \left(\complexI\rho^{\otimes N}_{\theta_\eps}\left(\Phi_{N,\theta_\eps,w_\eps}(t)-\Phi_{N,\theta_\eps,w_\eps}(0)\right),\rho_{\theta_\eps}^{\otimes N}\Phi\right) &= \int_0^t \left(\complexI\mathcal{H}_{N,\theta_\eps,w_\eps}\rho^{\otimes N}_{\theta_\eps}\Phi_{N,\theta_\eps,w_\eps}(s),\rho_{\theta_\eps}^{\otimes N}\Phi\right)\d s\\
        &= \int_0^t\sum_{j=1}^N \re\int_{U^N}\left(\nabla_{x_j}\Phi_{N,\theta_\eps,w_\eps}\cdot\overline{\nabla_{x_j}\Phi}\right)\rho_{\theta_\eps}(x_j)^2\d x\d s\\
        &+ \int_0^t\sum_{j=1}^N \re\int_{U^N}\left(|x_j|^2(1-Y_\eps(x_j))\Phi_{N,\theta_\eps,w_\eps}\overline{\Phi}\right)\rho_{\theta_\eps}(x_j)^2\d x\d s\\
        &- \int_0^t\sum_{j=1}^N \left\langle Z_{\eps}(x_j) \rho_{\theta_\eps}^{\otimes N}\Phi_{N,\theta_\eps,w_\eps}, \rho_{\theta_\eps}^{\otimes N}\Phi\right\rangle\d s\\
        &- \int_0^t\frac{1}{N}\sum_{1\leqslant i<j\leqslant N}\re\int_{U^N} w_\eps(x_i-x_j)\Phi_{N,\theta_{\eps},w_\eps}(x)\overline{\Phi(x)}\left(\rho_{\theta_\eps}^{\otimes N}(x)\right)^2 \d x\d s.
    \end{align*}
    Now, remark that all terms inside the time integrals are bounded uniformly in $\eps$ by \cref{Eq:BoundPhiNtheta,Lem:ProductRule,Cor:ProductRuleNegPosReg,Prop:action_d/dx,Cor:action_V,Eq:GenManyBodyQuasiCoercivEstimate}. Thus, by dominated convergence, it is sufficient to prove each term converges for almost all fixed times. First, as $Y_\eps$ converges to $Y$ in $\H^{2-\frac{d}{2}-,\infty}$, $\rho_{\theta_\eps}$ converges to $\rho_{\theta}$ in $\H^{2-\frac{d}{2}-,\infty}$. Thus,
    $$\sum_{j=1}^N \int_{U^N}\nabla_{x_j}\Phi_{N,\theta_\eps,w_\eps}(t)\cdot\overline{\nabla_{x_j}\Phi}\rho_{\theta_\eps}(x_j)^2\d x $$
    converges to 
    $$\sum_{j=1}^N \int_{U^N}\nabla_{x_j}\Phi_{N}(t)\cdot\overline{\nabla_{x_j}\Phi}\rho_{\theta}(x_j)^2\d x $$
    for almost all $t\in[-T,T]$, as $\nabla_{x_j}\Phi_{N,\theta_\eps,w_\eps}(t)$ converges weakly to $\nabla_{x_j}\Phi_{N}(t)$ in $\L^2$ for almost all $t\in[-T,T]$ and $\rho_{\theta_\eps}(x_j)^2 \nabla_{x_j} \Phi$ converges strongly to $\rho_{\theta}(x_j)^2 \nabla_{x_j} \Phi$ in $\L^2$. Similarly,
    $$\sum_{j=1}^N \int_{U^N}|x_j|^2(1-Y_\eps(x_j))\Phi_{N,\theta_\eps,w_\eps}(t)\overline{\Phi}\rho_{\theta_\eps}(x_j)^2\d x $$
    converges to 
    $$\sum_{j=1}^N \int_{U^N}|x_j|^2(1-Y(x_j))\Phi_{N}(t)\overline{\Phi}\rho_{\theta}(x_j)^2\d x $$
    for almost all $t\in[-T,T]$, as $x_j\Phi_{N,\theta_\eps,w_\eps}(t)$ converges weakly to $x_j\Phi_{N}(t)$ in $\L^2$ for almost all $t\in[-T,T]$ and $(1-Y_{\eps}(x_j))\rho_{\theta_\eps}(x_j)^2 x_j \Phi$ converges strongly to $(1-Y(x_j))\rho_{\theta}(x_j)^2 x_j \Phi$ in $\L^2$. Moreover, 
    $$ \sum_{j=1}^N \left\langle Z_{\eps}(x_j) \rho_{\theta_\eps}^{\otimes N}\Phi_{N,\theta_\eps,w_\eps}(t), \rho_{\theta_\eps}^{\otimes N}\Phi\right\rangle$$
    converges to
    $$ \sum_{j=1}^N \left\langle Z(x_j) \rho_{\theta}^{\otimes N}\Phi_{N}(t), \rho_{\theta}^{\otimes N}\Phi\right\rangle,$$
    by \cref{Lem:ProductRule,Cor:ProductRuleNegPosReg}, as $\Phi_{N,\theta_\eps,w_\eps}(t)$ converges strongly in $\W^{0+,2}$, $Z_\eps$ converges strongly in $\W^{0-,\infty}$, $\rho_{\theta}^{\otimes N}$ converges strongly in $\H^{0+,\infty}$ and $\Phi\in\W^{0+,2}$. Finally, as $w_\eps$ converges strongly to $w$ in $\mathfrak{W}$, thus in $\L^2+\L^\infty$, $\Phi_{N,\theta_\eps,w_\eps}(t)$ converges strongly in $\W^{\frac{3}{4},2}\subset\L^2_{\hat{x}_j}(\L^2\cap\L^4)_{x_j}$, by Sobolev embeddings, and $\rho_{\theta_\eps}$ converges strongly in $\L^\infty$, it follows
    $$\frac{1}{N}\sum_{1\leqslant i<j\leqslant N}\int_{U^N} w_\eps(x_i-x_j)\Phi_{N,\theta_{\eps},w_\eps}(t,x)\overline{\Phi(x)}\left(\rho_{\theta_\eps}^{\otimes N}(x)\right)^2 \d x$$
    converges to
    $$\frac{1}{N}\sum_{1\leqslant i<j\leqslant N}\int_{U^N} w(x_i-x_j)\Phi_{N}(t,x)\overline{\Phi(x)}\left(\rho_{\theta}^{\otimes N}(x)\right)^2 \d x.$$
    Thus, by dominated convergence, $\rho_\theta^{\otimes N}\Phi_N$ is a solution of \cref{Eq:ManyBodySchrodinger} associated to $\theta$ and $w$ starting from $\Psi_N^0$. By uniqueness, we deduce $\Psi_{N,\theta_\eps,w_\eps}$ converges to $\Psi_{N,\theta,w}$ in $\L^\infty_{loc}(\R,(\L^2)^{\otimes N})$.

\end{proof}

\subsection{Proof of \cref{Th:MeanFieldLimitAndersonHartree}}

It is well-known (see for example \cite{PicklKnowles,RodnianskiQuantumFluctuations}) that for any $k\in\N^*$, there exists $C_k>0$ such that
\begin{equation}\label{Eq:ControlTraceNormAlpha}
    \Tr\left|\Gamma_{k,N,\theta,w}(t)-\left(|u^{\theta,w}(t)\rangle\langle u^{\theta,w}(t)|\right)^{\otimes k}\right|\leqslant C_k\alpha_{N,\theta,w}(t),
\end{equation}
where 
$$\alpha_{N,\theta,w}(t) = \left|q^{\theta,w}(t)\Psi_{N,\theta,w}(t)\right|_{\L^2}^2$$ with $q^{\theta,w}(t) = 1-|u^{\theta,w}(t,x_1)\rangle\langle u^{\theta,w}(t,x_1)|$. Remark that $\left(|u^{\theta,w}(t)\rangle\langle u^{\theta,w}(t)|\right)^{\otimes k}=\left|u^{\theta,w}(t)^{\otimes k}\right\rangle\left\langle u^{\theta,w}(t)^{\otimes k}\right|$ is a rank one projector, thus 
\begin{equation}\label{Eq:ControlTraceNormByOpNorm}
    \Tr\left|\Gamma_{k,N,\theta,w}(t)-\left(|u^{\theta,w}(t)\rangle\langle u^{\theta,w}(t)|\right)^{\otimes k}\right|\leqslant 2\left|\Gamma_{k,N,\theta,w}(t)-\left(|u^{\theta,w}(t)\rangle\langle u^{\theta,w}(t)|\right)^{\otimes k}\right|_{\mathcal{L}\left((\L^2)^{\otimes N}\right)}
\end{equation}
(see Remark 1.4. in \cite{RodnianskiQuantumFluctuations}). In view of the usual estimate
\begin{equation}\label{Eq:ControlOpNormByTraceNorm}
    \Tr\left|\Gamma_{k,N,\theta,w}(t)-\left(|u^{\theta,w}(t)\rangle\langle u^{\theta,w}(t)|\right)^{\otimes k}\right|\geqslant \left|\Gamma_{k,N,\theta,w}(t)-\left(|u^{\theta,w}(t)\rangle\langle u^{\theta,w}(t)|\right)^{\otimes k}\right|_{\mathcal{L}\left((\L^2)^{\otimes N}\right)},
\end{equation}
the trace-class convergence of marginals is equivalent to the convergence in operator norm.\\

To prove \cref{Th:MeanFieldLimitAndersonHartree},
it is sufficient to show $\alpha_{N,\theta,w}(t)$ converges to $0$ uniformly in $(t,\theta,w)$ in bounded sets of $\R\times\Theta\times\mathfrak{W}$. Now, remark that $\alpha_{N,\theta,w}(0)=0$ as $\Psi_{N,\theta,w}(0)=(u^{\theta,w}(0))^{\otimes N}$. Thus, by \cref{Prop:PicklKnowles,Lem:MassConservation}, for any bounded set $B\subset\R\times\Theta\times\mathfrak{W}$, it holds
\begin{equation}
    \forall (t,\theta,w)\in B,\; \alpha_{N,\theta,w}(t) \leqslant \exp\left(64\sup_{(t,\theta,w)\in B}|t|\left[w^2\right]\left|u^{\theta,w}\right|_{\L^\infty_t\mathcal{D}^{1,2}_x}^2\right) \frac{1}{N}.
\end{equation}
Thus, by \cref{Eq:ControlNormD12theta,Eq:Equivalence[w2]theta,Eq:ControlTraceNormAlpha}, for any $k\in\N^*$ and any bounded set $B\subset\R\times\Theta\times\mathfrak{W}$, 
$$\sup_{(t,\theta,w)\in B}\Tr\left|\Gamma_{k,N,\theta,w}(t)-\left(|u^{\theta,w}(t)\rangle\langle u^{\theta,w}(t)|\right)^{\otimes k}\right|=O\left(\frac{1}{N}\right)$$
as $N$ goes to infinity. Finally, \cref{Lem:ContinuityPsiNtheta} and the Moore-Osgood theorem imply \cref{Th:MeanFieldLimitAndersonHartree}.

\begin{rem}

    Our work can be adapted for many parameterized families of self-adjoint operators. For example, using the construction of Anderson operators of \cite{mouzard2023simple}, one can prove an analogue of \cref{Th:MeanFieldLimitAndersonHartree} for Anderson operators on the torus $\T^d$ ($d\in\{1,2,3\}$) without relying on Strichartz estimates proven in \cite{MouzardZachhuberStrichartz,Zachhuber3dAndersonHartreeMeanField}. Moreover, it gives a stronger result than \cite{Zachhuber3dAndersonHartreeMeanField} as it proves the mean-field limit to Anderson-Hartree equation for more general interaction potentials than the $\L^\infty$ case.
    
\end{rem}

\section*{Acknowledgement}

The author was supported by the ANR project Smooth ANR-22-CE40-0017. The author wants to express their deepest gratitude to Théo Hérouard for the interesting discussions they had on Pickl's method, which were the starting point of this work.

\printbibliography

\end{document}